\documentclass{amsart}
\usepackage{ xcolor} 

\usepackage{amsmath, amsthm}
\usepackage{amssymb}
\usepackage{array}
\usepackage{amsfonts}
\usepackage[symbol]{footmisc}
\usepackage{palatino}
\usepackage{nicefrac}
\usepackage{graphicx}
\usepackage{tikz-cd}
\usepackage{hyperref}
\usepackage[capitalise]{cleveref}
\usepackage{mathdots}
\usepackage{framed}
\usepackage{enumitem}
\usepackage{ytableau}
\usepackage{caption}

\crefname{equation}{Equation}{Equations}
\crefname{conjecture}{Conjecture}{Conjectures}
\crefname{figure}{Figure}{Figures}
\crefname{thm}{Theorem}{Theorems}
\crefname{theorem}{Theorem}{Theorems}

\newtheorem{thm}{Theorem}
\numberwithin{thm}{section} 
\newtheorem{theorem}[thm]{Theorem}
\newtheorem{lemma}[thm]{Lemma}
\newtheorem{cor}[thm]{Corollary}
\newtheorem{prop}[thm]{Proposition}

\theoremstyle{definition}
\newtheorem{defn}[thm]{Definition}

\theoremstyle{remark}
\newtheorem{rem}[thm]{Remark}

\newtheorem{example}[thm]{Example}

\newcommand{\Col}[2][I(n)]{\ensuremath{Col\begin{pmatrix}#1\\#2\end{pmatrix}}}

\newcommand{\e}[1]{\ensuremath{E_{#1}}}
\newcommand{\Fl}{\ensuremath{\mathcal Fl}}

\newcommand{\bo}{\ensuremath{\mathfrak b}}
\newcommand{\g}{\ensuremath{\mathfrak g}}
\newcommand{\p}{\ensuremath{\mathfrak p}}
\newcommand{\tr}{\ensuremath{\operatorname{tr}}}
\newcommand{\IC}{\mathbb C}

\newcommand{\ess}{\mathcal Ess}
\newcommand{\rk}{\operatorname{rk}}

\newcommand{\set}[2]{\left\{#1\,\middle\vert\,#2\right\}}

\begin{document}
\title{Conormal Varieties of Covexillary Schubert Varieties}
\author{Rahul Singh}
\address{ Institute for Computational and Experimental Research in Mathematics, Brown University, Providence, RI 02903 USA }
\email{rahul.sharpeye@gmail.com}

\begin{abstract} 
A permutation is called covexillary if it avoids the pattern $3412$.
We construct an open embedding of a covexillary matrix Schubert variety into a Grassmannian Schubert variety.
As applications of this embedding,
we show that the characteristic cycles of covexillary Schubert varieties are irreducible,
and provide a new proof of Lascoux's model computing Kazhdan-Lusztig polynomials of vexillary permutations.
Combining the above embedding with earlier work of the author on the conormal varieties of Grassmannian Schubert varieties,
we develop an algebraic criterion identifying the conormal varieties of covexillary Schubert and matrix Schubert varieties as subvarieties of the respective cotangent bundles.
\end{abstract}

\maketitle

The conormal varieties of Schubert varieties play a central role in the Springer representations of Weyl groups,
see \cite{MR491988,MR689533,MR1433132}.
The work of Aluffi, Mihalcea, Sh\"urmann, and Su \cite{AMSS:shadows} sheds light on the relation between the torus-equivariant (co)-homology classes of conormal varieties and the ``stable basis'' of Maulik and Okounkov \cite{MR3951025}.
The conormal variety is defined as the closure (in the cotangent bundle) of the conormal bundle of the smooth locus of a variety.
While the conormal bundle of the smooth locus often admits simple descriptions,
it is in general a difficult problem to describe the boundary, singularities, cohomology classes etc. of its closure.
In this paper, we study covexillary Schubert varieties and their conormal varieties.

Recall that a partial permutation is called covexillary if its essential set goes from top-right to bottom-left,
see \cref{covexPP} for details.
Following \cite{fulton:matrixSchubertVarieties}, a permutation, i.e.,
a full-rank partial permutation matrix, is covexillary if it avoids the pattern $3412$.
Lascoux and Sch\"utzenberger \cite{MR660739} (see also Fulton \cite{fulton:matrixSchubertVarieties})
showed that the cohomology class of a covexillary%
\footnote[2]{We use the `homological' indexing, corresponding to $\dim(\Fl_w)=\ell(w)$, where $\ell(w)$ is the length of the word $w$.
In particular, the Schubert varieties that we call covexillary are called vexillary in \cite{MR660739,fulton:matrixSchubertVarieties,MR1354702} etc.}
Schubert variety is a multi-Schur function,
generalizing Giambelli's formula for the cohomology class of a Grassmannian Schubert variety.
More recently, Anderson, Ikeda, Jeon, and Kawago \cite{andersonIkedaJeonKawago}
have shown that covexillary Schubert varieties have the same singularities as Schubert subvarieties of Grassmannians.
A priori, it is not clear why covexillary Schubert varieties behave like Grassmannian Schubert varieties for both these properties (cohomology classes and singularity types).
In this paper, we provide a construction (see \cref{mainEmbed} below) that simultaneously explains both these properties.

Let $B$ denote the set of upper triangular invertible $n\times n$ matrices,
\g\ the set of $n\times n$ matrices,
$w$ a partial permutation matrix,
and $\g_w=\overline{BwB}$ the closure (in \g) of the $B\times B$-orbit of $w$,
i.e., a matrix Schubert variety.
Our first result (\cref{prop:MSVtoGr}) is the construction,
for covexillary $w$,
of an open embedding of $\g_w$ in some Schubert subvariety of the Grassmannian $Gr(n,2n)$.

\begin{theorem}
\label{mainEmbed}
Let $w$ be a covexillary permutation.
Let $h:\g\to Gr(n,2n)$ be the graph embedding, i.e.,
the map given by $x\mapsto Col\begin{pmatrix}I(n)\\x\end{pmatrix}$.
There exists a permutation $\tau\in S_{2n}$ and a Schubert variety $Gr_{\underline v}\subset Gr(n,2n)$
 such that the composite map
\begin{align*}
\tau\circ h:\g\to Gr(n,2n),&&x\mapsto \tau\,Col\begin{pmatrix}I(n)\\x\end{pmatrix}
\end{align*}
induces an open immersion
$\g_w\hookrightarrow Gr_{\underline v}$.
\end{theorem}

Let us explain how \cref{mainEmbed} relates the singularities and cohomology classes of covexillary Schubert varieties with those of Grassmannian Schubert varieties.
Let $E_\bullet=(E_1\subset\cdots\subset E_{n-1})$ denote the standard flag,
$
\Fl 
$
the flag variety,
and $\pi:G\twoheadrightarrow\Fl$ the map $g\mapsto gE_\bullet$.
Let $\Fl_w=\overline{BwE_\bullet}$ denote the Schubert variety corresponding to a permutation matrix $w$,
and $G_w:=\pi^{-1}(\Fl_w)$ the pull-back of $\Fl_w$ along $\pi$.
Observe that since $\pi$ is a locally trivial fibration with smooth fibres,
$\Fl_w$ has the same singularities as $G_w$.
Further, $G_w$ is an open subset of $\g_w$.
Now, if $w$ is covexillary,
it follows from \cref{mainEmbed} that $\g_w$ (and hence also $G_w$) can be viewed as an open subset of $Gr_{\underline v}$,
and hence has the same singularities as $Gr_{\underline v}$.

Let $w_0$ be the longest permutation in $S_n$,
and let $T\subset B$ be the set of diagonal matrices in $B$.
Recall that the double Schubert polynomial $\mathfrak S_{w_0w}$,
which computes the $T$-equivariant cohomology class of $\Fl_w$,
is precisely the $T\times T$ multidegree of $\g_w$ in $\g$,
i.e., the localization of the $T\times T$-equivariant cohomology class of $\g_w$ at $0$.
Let $T_{2n}$ denote the set of diagonal matrices in $GL_{2n}$, acting on $Gr(n,2n)$ by left multiplication.
The embedding $\g_w\to Gr_{\underline v}$ sends $0\in\g_w$ to $\tau\in Gr(n,2n)$
and is equivariant for an appropriate identification $T\times T\overset\sim\to T_{2n}$
(the precise identification depends on the word $w$).
This relates the $T\times T$ multidegree of $\g_w$ in $\g$,
i.e., the double Schubert polynomial $\mathfrak S_{w_0w}$,
with the $T_{2n}$-equivariant localization of $Gr_{\underline v}$ at the point $\tau$,
which is a multi-Schur function.

The fact that covexillary Schubert varieties have the same singularities as Grassmannian Schubert varieties yields some immediate consequences.
The characteristic cycle of the IC sheaf of a Schubert variety $\Fl_w$ is irreducible
if and only if for each $T$-fixed point $v$,
the local Euler obstruction of $\Fl_w$ at $v$ equals the evaluation of the
Kazhdan-Lusztig polynomial \cite{KL:topological,MR573434} evaluated at $1$.
Both the Kazhdan-Lusztig polynomials and the Euler obstructions are local invariants,
and hence we obtain the following result (\cref{mainChar}).
\begin{theorem}
\label{introChar}
The characteristic cycle of the IC sheaf of a covexillary Schubert variety is irreducible.
\end{theorem}

In \cite{MR646823}, Lascoux and Sch\"utzenberger gave a combinatorial model
to compute the Kazhdan-Lusztig polynomial of a Grassmannian Schubert variety.
Later, Lascoux \cite{MR1354702} extended this to covexillary Schubert varieties,
showing that the Kazhdan-Lusztig polynomials of a covexillary Schubert variety
are the same as certain Kazhdan-Lusztig polynomials of some Grassmannian Schubert variety.
Since the Kazhdan-Lusztig polynomials are local invariants,
\cref{mainEmbed} gives an alternate proof of Lascoux's result.
Further, by \cref{introChar}, the Lascoux-Sch\"utzenberger model is also an effective algorithm 
for computing the Euler obstructions of covexillary Schubert varieties.

Finally, we turn our attention to the conormal variety $N^*\Fl_w$ of $\Fl_w$ in $\Fl$,
which we relate to the conormal variety $N^*Gr_{\underline v}$ of $Gr_{\underline v}$ in $Gr(n,2n)$ via the commutative diagram in \cref{mainDiag}.
Here $f_{\#}:T^*M\to T^*N$ denotes the open immersion induced on cotangent bundles
from an open immersion $f:M\to N$ of smooth varieties.
The existence of a map $\pi':N^*G_w\twoheadrightarrow N^*\Fl_w$
is a consequence of the smoothness of the map $\pi:G\to\Fl$, see \cite{HTT}.

\begin{equation}
\label{mainDiag}
\begin{tikzcd}
N^*\Fl_w\arrow[d]\arrow[dr,phantom,"\times"]     & N^*G_w\arrow[d]\arrow[l,"\pi'"',->>]\arrow[dr,phantom,"\times"]\arrow[r,hook,"\iota_{\#}"]\arrow[d] & N^*\g_w\arrow[d]\arrow[r,"\tau_{\#}\circ h_{\#}"]\arrow[dr,phantom,"\times"] & N^*Gr_{\underline v}\arrow[d]\\
\Fl_w\arrow[d,hook]  \arrow[dr,phantom,"\times"] & \arrow[l,"\pi"',->>]G_w\arrow[d,hook] \arrow[dr,phantom,"\times"]\arrow[r,hook,"\iota"]             & \g_w\arrow[d,hook]\arrow[r,hook,"\tau\circ h"]\arrow[dr,phantom,"\times"]    & Gr_{\underline v}\arrow[d,hook]\\
\Fl                                              & G\arrow[l,"\pi"',->>]\arrow[r,hook,"\iota"]                                                         & \g\arrow[r,hook,"\tau\circ h"]                                               & Gr(n,2n)
\end{tikzcd}
\end{equation}

In \cite{MR4257099}, algebraic conditions, equivalently a (possibly non-reduced) system of equations,
was developed identifying the conormal variety of a Grassmannian Schubert variety as a closed subvariety of the cotangent bundle of the Grassmannian.
Following \cref{mainDiag},
the conormal variety $N^*\g_w$ of the matrix Schubert variety $\g_w$
is simply the pull-back of the conormal variety $N^*Gr_{\underline v}$
along the map $\tau_{\#}\circ h_{\#}:T^*\g_w\to T^*Gr(n,2n)$,
or equivalently,
$N^*\g_w=N^*Gr_{\underline v}|\g_w$.
Using this characterization of $N^*\g_w$,
along with the results of \cite{MR4257099},
we develop a description of $N^*\g_w$ as a subvariety of $T^*\g$. 
(see \cref{thm:conMat} for a precise statement).

Using the trace form, we make the identifications $\g^*=\g$ and $T^*\g=\g\times\g$.

\begin{theorem}
\label{mainMat}
Suppose $w$ is covexillary.
A point $(x,y)\in\g\times\g= T^*\g$ belongs to the conormal variety $N^*\g_w$ 
of $\g_w$
if and only if it certain rank conditions are satisfied on certain sub-matrices of the $2n\times 2n$ matrix
$\begin{pmatrix}
yx&y\\xyx&xy\\
\end{pmatrix}
$.
\end{theorem}

Consider the pull-back $\pi^*T^*\Fl$ of $T^*\Fl$ along $\pi:G\to\Fl$,
and the induced map $\pi':\pi^*T^*\Fl\to T^*\Fl$.
A point $p\in T^*\Fl$ belongs to the conormal variety $N^*\Fl_w$ if and only if $\pi'^{-1}(x)\in N^*G_w$.
Moreover, the latter is precisely the restriction of $N^*\g_w$ along the inclusion $G_w\to\g_w$. 
This allows us to develop a characterization of $N^*\Fl_w$.

\begin{theorem}
\label{mainSchub}
Suppose $w$ is covexillary,
and let $1\leq p_1\leq\cdots\leq p_m<n$,
$1\leq q_1\leq\cdots\leq q_m< n$
and $r_1,\cdots,r_m$ be integers
such that
\begin{align*}
&&&&\dim(F_{q_i}/E_{p_i})\leq r_i,&&\forall\,1\leq i\leq m&&
\end{align*}
are the minimal conditions describing $\Fl_w$ as a subvariety of $\Fl$ (see \cref{lemEss}).
Consider the cotangent bundle,
\begin{align*}
T^*\Fl=\set{(F_\bullet,z)\in\Fl\times\g}{zF_i\subset F_{i-1}\forall i}.
\end{align*}
Then $(F_\bullet,z)\in N^*\Fl_w$ if and only if $F_\bullet\in\Fl_w$,
and
\begin{align*}
\dim(z(F_{q_i}+E_{p_i})/(F_{q_j}\cap E_{p_j}))\leq 
\begin{cases}
(q_{i-1}-r_{i-1})-(q_j-r_j),\\
(p_i+r_i)-(p_{j+1}+r_{j+1}),\\
\end{cases}
\forall 1\leq j<i\leq n.
\end{align*}
\end{theorem}

\emph{Acknowledgements}: The author would like to thank Leonardo Mihalcea for many fruitful and illuminating conversations.
In particular, the implications of \cref{mainEmbed} for characteristic cycles and Kazhdan-Lusztig polynomials arose out of these conversations.
This work was conducted while the author enjoyed the hospitality and the wonderful environment of ICERM, 
for which he is grateful.

\section{Matrix Schubert Varieties and the Graph Embedding}
\label{secMSV}
Let $\e{n}$ be a $n$-dimensional vector space with standard basis $e_1,\cdots,e_{n}$.
For linear subspaces $V,\,V'\subset \e{n}$,
we will denote by $V/V'$ the image of $V$ under the projection $\e{n}\to \e{n}/V'$.

Let $\g$ be the set of $n\times n$ matrices,
$G=GL_n$ the set of invertible matrices in $\g$,
and $B$ the set of upper triangular matrices in $G$.
The algebra $\g$ (resp. the group $G$) acts on $\e{n}$ by left multiplication with respect to the ordered basis $e_1,\cdots,e_n$.

\subsection{Partial Permutations and the rank matrix}
A partial permutation (of size $n$) is a $\{0,1\}$-matrix (of size $n\times n$)
such that each row and each column contains at most one non-zero entry.
A full-rank partial permutation matrix contains exactly one non-zero entry in each row and column,
and hence corresponds to a permutation of the set $\{1,\cdots,n\}$.
(The permutation $w\in S_n$ corresponds to the $n\times n$ matrix 
with $1$s in the positions $\set{(w(i),i)}{1\leq i\leq n}$, and $0$s elsewhere).

Given a partial permutation $w$ of size $n\times n$,
we define a \emph{rank matrix} $\mathbf r^w$ given by
\begin{align*}
\mathbf r^w_{ij}=\dim (w\e j/\e{i-1}).
\end{align*}
Equivalently, $\mathbf r^w_{ij}$ is the rank of the bottom-left sub-matrix of $w$
spanned by the rows $i,i+1,\cdots,n$, and the columns $1,\cdots,j$.

\begin{example}
Consider the permutation $w\in S_6$ given in one-line notation by $w=[351642]$.
The corresponding permutation and rank matrices are given by
\begin{align*}
w=
\begin{pmatrix}
0&0&1&0&0&0\\
0&0&0&0&0&1\\
1&0&0&0&0&0\\
0&0&0&0&1&0\\
0&1&0&0&0&0\\
0&0&0&1&0&0\\
\end{pmatrix}
,&&
\mathbf r^w=
\begin{pmatrix}
1&2&3&4&5&6\\
1&2&2&3&4&5\\
1&2&2&3&4&4\\
0&1&1&2&3&3\\
0&1&1&2&2&2\\
0&0&0&1&1&1\\
\end{pmatrix}
.
\end{align*}
\end{example}

\subsection{Diagram and Essential Set}
\label{covexPP}
Consider a partial permutation $w$ of size $n$.
For each non-zero entry in $w$, we shade every box that appears above it and every box that appears to its right.
The set of unshaded boxes is called the \emph{diagram} $D(w)$ of $w$.
The set of northeast corners of $D(w)$ is called the \emph{essential set} $\ess(w)$ of $w$,
\begin{equation*}
\ess(w)=\set{(p,q)\in D(w)}{(p-1,q),(p,q+1),(p-1,q+1)\notin D(w)}.
\end{equation*}
It is clear that the essential set does not contain any boxes in the top row or the last column.

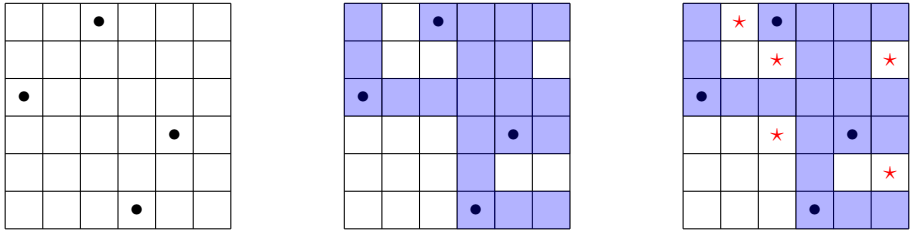
\begin{figure}[ht]
\begin{tikzpicture}[>=latex,line join=bevel,scale=0.5]
\foreach \x in {0,...,6}
{
    \draw (0.5,-0.5-\x)--(6.5,-0.5-\x);
    \draw (\x+0.5,-0.5)--(\x+0.5,-6.5);
}

\draw (1,-3) node {$\bullet$};
\draw (3,-1) node {$\bullet$};
\draw (4,-6) node {$\bullet$};
\draw (5,-4) node {$\bullet$};
\end{tikzpicture}
\qquad\qquad
\begin{tikzpicture}[>=latex,line join=bevel,scale=0.5]
\foreach \x in {0,...,6}
{
    \draw (0.5,-0.5-\x)--(6.5,-0.5-\x);
    \draw (\x+0.5,-0.5)--(\x+0.5,-6.5);
}

\draw (1,-3) node {$\bullet$};
\draw (3,-1) node {$\bullet$};
\draw (4,-6) node {$\bullet$};
\draw (5,-4) node {$\bullet$};

\filldraw [fill=blue,opacity=.3] (0.5,-3.5) rectangle (1.5,-0.5);
\filldraw [fill=blue,opacity=.3] (3.5,-6.5) rectangle (4.5,-0.5);
\filldraw [fill=blue,opacity=.3] (4.5,-4.5) rectangle (5.5,-0.5);
\filldraw [fill=blue,opacity=.3] (2.5,-0.5) rectangle (3.5,-1.5);
\filldraw [fill=blue,opacity=.3] (5.5,-0.5) rectangle (6.5,-1.5);
\filldraw [fill=blue,opacity=.3] (1.5,-2.5) rectangle (3.5,-3.5);
\filldraw [fill=blue,opacity=.3] (4.5,-5.5) rectangle (6.5,-6.5);
\filldraw [fill=blue,opacity=.3] (5.5,-2.5) rectangle (6.5,-4.5);
\end{tikzpicture}
\qquad\qquad
\begin{tikzpicture}[>=latex,line join=bevel,scale=0.5]
\foreach \x in {0,...,6}
{
    \draw (0.5,-0.5-\x)--(6.5,-0.5-\x);
    \draw (\x+0.5,-0.5)--(\x+0.5,-6.5);
}
\draw (1,-3) node {$\bullet$};
\draw (3,-1) node {$\bullet$};
\draw (4,-6) node {$\bullet$};
\draw (5,-4) node {$\bullet$};

\filldraw [fill=blue,opacity=.3] (0.5,-3.5) rectangle (1.5,-0.5);
\filldraw [fill=blue,opacity=.3] (3.5,-6.5) rectangle (4.5,-0.5);
\filldraw [fill=blue,opacity=.3] (4.5,-4.5) rectangle (5.5,-0.5);
\filldraw [fill=blue,opacity=.3] (2.5,-0.5) rectangle (3.5,-1.5);
\filldraw [fill=blue,opacity=.3] (5.5,-0.5) rectangle (6.5,-1.5);
\filldraw [fill=blue,opacity=.3] (1.5,-2.5) rectangle (3.5,-3.5);
\filldraw [fill=blue,opacity=.3] (4.5,-5.5) rectangle (6.5,-6.5);
\filldraw [fill=blue,opacity=.3] (5.5,-2.5) rectangle (6.5,-4.5);

\draw[red] (2,-1) node{$\star$};
\draw[red] (3,-2) node{$\star$};
\draw[red] (3,-4) node{$\star$};
\draw[red] (6,-2) node{$\star$};
\draw[red] (6,-5) node{$\star$};
\end{tikzpicture}
\caption{Computing the diagram and essential set of a partial permutation matrix.}
\end{figure}

Given a partial permutation $w$,
let $\widehat w$ be the unique $2n\times 2n$ permutation matrix whose bottom-left $n\times n$ submatrix is $w$,
and the entries in the top $n$ rows and in the last $n$ columns go from bottom-left to top-right.
Then the essential set of $\widehat w$ is contained in the bottom-left $n\times n$ submatrix,
and agrees with the essential set of $w$.
The significance of essential sets can be gleaned from the following result.

\begin{lemma}
\label{lemEss}
(cf. \cite{fulton:matrixSchubertVarieties})
The partial permutation matrix $w$ is determined by the pair $(\ess(w),\mathbf r^w|\ess(w))$,
i.e., the essential set of $w$, and the restriction of the rank matrix to this essential set.
\end{lemma}
An algorithm to recover $w$ from the pair $(\ess(w),\mathbf r^w|\ess(w))$ was developed by Eriksson and Linusson in \cite{MR1412437}.

\begin{defn}
We say that a partial permutation $w$ is \emph{covexillary} if its essential set goes from the top-left to the bottom-right,
i.e., there exist integers $1\leq p_1\leq\cdots\leq p_{m-1}< n$ and $1\leq q_1\leq\cdots\leq q_{m-1}< n$
such that
\begin{align*}
\ess(w)=\{(p_1+1,q_1),\cdots, (p_{m-1}+1,q_{m-1})\}.
\end{align*}
\end{defn}

\begin{lemma}
(cf. \cite{fulton:matrixSchubertVarieties})
A permutation $w\in S_n$ is covexillary if and only if it avoids the pattern $3412$,
i.e., there does not exist any quadruple $1\leq i<j<k<l\leq n$ satisfying $w(k)<w(l)<w(i)<w(j)$.
\end{lemma}

\subsection{Schubert varieties in the Flag Manifold}
Let $E_i=\left\langle e_1,\cdots,e_n\right\rangle$,
let $E_\bullet=(E_1\subset\cdots\subset E_{n-1})$ denote the standard flag in $E_n$,
and let \Fl\ denote the variety of flags in $\e{n}$, i.e.,
\begin{align*}
\Fl=\set{(F_1\subset\cdots\subset F_{n-1})}{\dim F_i=i}.
\end{align*}
For $w\in S_n$,
the $B$-orbit $\Fl_w^\circ$ of the permutation flag $wE_\bullet$ in $\Fl$ is called a Schubert cell,
and its closure $\Fl_w$ a Schubert variety.
We have
\begin{equation}
\label{SchubertConditions}
\begin{split}
\Fl^\circ_w & =\set{F_\bullet\in\Fl}{\dim(F_{j}/\e{i-1})=\mathbf r_w(i,j)\ \forall\ 1\leq i,j\leq n},\\
\Fl_w       & =\set{F_\bullet\in\Fl}{\dim(F_{j}/\e{i-1})\leq\mathbf r_w(i,j)\ \forall\ 1\leq i,j\leq n}.
\end{split}
\end{equation}
Given $u,w\in S_n$, we have $\Fl_u\subset\Fl_w$ if and only if
\begin{align*}
&&\mathbf r_u(i,j)\leq\mathbf r_w(i,j)&&\forall\,1\leq i,j\leq n.
\end{align*}
The partial order $u\leq w\iff\Fl_u\subset\Fl_w$ is called the Bruhat order on $S_n$.

\begin{rem}
A permutation which avoids the pattern $2143$ is called vexillary.
Let $w_0$ denote the longest permutation in $S_n$.
Then $w\in S_n$ is covexillary if and only if $w_0w$ is vexillary.
There are two common conventions for indexing Schubert varieties by permutations:
one corresponds to $\dim (\Fl^w)=\ell(w_0)-\ell(w)$,
the other to $\dim(\Fl_w)=\ell(w)$,
where $\ell(w)$ is the length of $w$.
The Schubert varieties indexed by vexillary permutations in the former convention
correspond to the ones indexed by covexillary permutations in the latter.
\end{rem}

\subsection{Grassmannian Schubert varieties}
Let $Gr(d,n)$ denote the Grassmannian variety of $d$-dimensional subspaces of $\e n$.
The $B$-orbit closures $Gr_{\underline u}\subset Gr(d,n)$ (called Grassmannian Schubert varieties)
are indexed by increasing $d$-sequences $\underline u$ taking values in $\{1,\cdots,n\}$.
Given such a sequence $\underline u=(u_1,\cdots,u_d)$,
we have
\begin{align}
\label{defGrassSchub}
Gr_{\underline u}=\set{V\subset \e n}{\dim(V+\e{u_i})\leq d+u_i-i\ \forall\,1\leq i\leq d}.
\end{align}
We have $Gr_{\underline u}\subset Gr_{\underline v}$ if and only if $u_i\leq v_i$ for all $1\leq i\leq d$.

Note the redundancy inherent in the conditions in \cref{defGrassSchub}.
Given a sequence $\underline u$ as above,
let $S_{\underline u}=\set{i\in[d]}{u_{i+1}\neq u_i+1}$.
Then
\begin{align*}
Gr_{\underline u}=\set{V\subset \e n}{\dim(V+\e{u_i})\leq d+u_i-i\ \forall\,i\in S_{\underline u}}.
\end{align*}

\subsection{Matrix Schubert varieties}
Consider the $B\times B$-action on $\g$ given by
\begin{align*}
(b_l,b_r)\cdot x =b_l x b_r^{-1}.
\end{align*}
The orbits $\g^\circ_w\subset \g$ of this action are indexed by $n\times n$ partial permutations.
Here $\g^\circ_w$ denotes the $B\times B$-orbit in \g\ of the partial permutation matrix $w$.
We call $\g^\circ_w$ a \emph{matrix Schubert cell},
and its closure $\g_w:=\overline{\g_w^\circ}$ a \emph{matrix Schubert variety}.
We have
\begin{equation}
\label{msvConditions}
\begin{split}
\g^\circ_w & =\set{x\in \g}{\dim(x\e j/\e {i-1})=\mathbf r_w(i,j)\ \forall\ 1\leq i,j\leq n},\\
\g_w       & =\set{x\in \g}{\dim(x\e j/\e {i-1})\leq\mathbf r_w(i,j)\ \forall\ 1\leq i,j\leq n}.
\end{split}
\end{equation}
Given partial permutations $u,w$, we have $\g_u\subset\g_w$ if and only if
\begin{align*}
&&\mathbf r_u(i,j)\leq\mathbf r_w(i,j)&&\forall i,j.
\end{align*}
The partial order $u\leq w\iff \g_u\subset\g_w$ (also called the Bruhat order) extends the Bruhat order on permutations. 

The set $G\subset \g$ is a $B\times B$-stable open subvariety of $\g$.
We have $\g^\circ_w\subset G$ if and only if $w$ is a permutation matrix,
i.e., if and only if $\rk(w)=n$.
In this case,
we set
\begin{align*}
G_w=\g_w\bigcap G=\bigcup\limits_{\substack{u\in S_n\\u\leq w}}\g^\circ_u.
\end{align*}

Consider the map $\pi:G\twoheadrightarrow\Fl$, given by $g\mapsto gE_\bullet$.
Comparing \cref{SchubertConditions,msvConditions}, we see that $G_w=\pi^{-1}(\Fl_w)$.

\subsection{Open immersion for covexillary matrix Schubert varieties}
Let $w$ be a covexillary partial permutation,
so that
\begin{align*}
\ess(w)=\{(p_1+1,q_1),\cdots,(p_{m-1}+1,q_{m-1})\}
\end{align*}
for some
$0\leq p_1\leq\cdots\leq p_{m-1}\leq n-1$ and $1\leq q_1\leq\cdots\leq q_{m-1}\leq n$.
Let $\mathbf r^w$ be the rank matrix of $w$,
and set $r_i=\mathbf r^w_{p_i+1,q_i}$.
For convenience, we set $p_m=q_m=n$, and $t_i=p_i+q_i$.
Let $\tau$ be the $2n\times 2n$ permutation matrix given by
\begin{align}
\label{tauMatrix}
\tau=
\begin{pmatrix}
I(q_1)&0&0&0&0&0\\
0&0&0&I(p_1)&0&0\\
0&\ddots&0&0&0&0\\
0&0&0&0&\ddots&0\\
0&0&I(q_m-q_{m-1})&0&0&0\\
0&0&0&0&0&I(p_m-p_{m-1})\\
\end{pmatrix},
\end{align}
where $I(k)$ denotes the identity matrix of size $k$, and the $0$s are rectangular zero matrices of appropriate sizes.
By an abuse of notation, we will also denote by $\tau:Gr(n,2n)\to Gr(n,2n)$ the automorphism of $Gr(n,2n)$
induced by the left multiplication action of $\tau$ on \e n. 

Let $Col(\ )$ denote the column span of a matrix,
and consider the map 
\begin{align}
\label{map:h}
h:\g\to Gr(n,2n),&&
h(x)=Col
\begin{pmatrix}
I_n\\
x\\
\end{pmatrix}.
\end{align}
If we view $x\in\g$ as a map $x:E_n\to E_n$,
then $h(x)$ can be identified as the graph of this map.
Hence we call $h:\g\to Gr(n,2n)$ the graph embedding.

\begin{theorem}
\label{prop:MSVtoGr}
Let $w\in S_n$ and $\tau\in S_{2n}$ be as in \cref{tauMatrix}.
Let $Gr_{\underline v}$ be the Schubert subvariety of $Gr(n,2n)$ given by
\begin{align*}
Gr_{\underline v}=\set{V\in Gr(n,2n)}{\dim(V+\e{t_i})\leq n+p_i+r_i}.
\end{align*}
The composite map $\tau\circ h:\g\to Gr(n,2n)$ induces an open immersion $\g_w\hookrightarrow Gr_{\underline v}$.
\end{theorem}
The proof of \cref{prop:MSVtoGr} depends on the following lemma.

\begin{lemma}
\label{lem:1rankCond}
Let $V=\left\langle e_1,\cdots e_q, e_{n+1},\cdots,e_{n+p}\right\rangle$.
For $x\in\g$,
we have 
\begin{align*}
\dim(x\e q/\e p)\leq r\iff\dim(h(x)+V)\leq n+p+r.
\end{align*}
\end{lemma}
\begin{proof}
Let us write $x$ in block-matrix form,
\begin{align*}
x=\begin{pmatrix}
a & b\\
c & d\\
\end{pmatrix},
\end{align*}
with $a$ a $p\times q$ matrix,
$b$ a $p\times (n-q)$ matrix,
$c$ a $(n-p)\times q$ matrix,
and $d$ a $(n-p)\times (n-q)$ matrix.
The result follows from the observations
$\dim(x\e q/\e p)=\rk(c)$,
and 
\begin{align*}
\dim (h(x)+V)
&=\rk
\begin{pmatrix}
I(q)& 0   & I(q)& 0\\
0   & I(p)& 0   & 0\\
a   & b   & 0   & I(p)\\
c   & d   & 0   & 0\\
\end{pmatrix}\\
&=\rk
\begin{pmatrix}
0 & 0   & I(q)& 0\\
0 & I(p)& 0   & 0\\
0 & 0   & 0   & I(p)\\
c & 0   & 0   & 0\\
\end{pmatrix}
=n+p+\rk(c).
\end{align*}
\end{proof}

\begin{proof}[Proof of \cref{prop:MSVtoGr}]
Since $h:\g\hookrightarrow Gr(n,2n)$ is an open immersion,
and $V\mapsto \tau V$ is an automorphism of $Gr(n,2n)$,
the composite map $\tau\circ h:\g\to Gr(n,2n)$ is an open immersion. 
Hence, it suffices to show that for $x\in\g$,
we have $\tau h(x)\in Gr_{\underline v}$ if and only if $x\in\g_w$.

Observe that $\tau h(x)\in Gr_{\underline v}$ if and only if
\begin{align*}
      & \dim(\tau h(x)+\e{t_i})\leq n+p_i+r_i     && \forall\,1\leq i\leq m-1\\
\iff  & \dim(h(x)+\tau^{-1}\e{t_i})\leq n+p_i+r_i && \forall\,1\leq i\leq m-1.
\end{align*}
Following \cref{tauMatrix}, we see that
\begin{align}
\label{tauInverseE}
\tau^{-1}\e{t_i}=\left\langle e_1,\cdots,e_{q_i},e_{n+1},\cdots,e_{n+p_i}\right\rangle.
\end{align}
Following \cref{lem:1rankCond}, we have
\begin{align*}
    &\dim(h(x)+\tau^{-1}\e{t_i})\leq n+p_i+r_i && \forall 1\leq i\leq m-1\\
\iff &\dim(x\e{q_i}/\e{p_i})\leq r_i && \forall 1\leq i\leq m-1.
\end{align*}
Comparing with \cref{msvConditions},
we see that 
$\tau h(x)\in Gr_{\underline v}$ if and only if $x\in\g_w$.
\end{proof}

\begin{cor}
Consider the open immersion $\iota:G\to\g$.
If $w\in S_n$ is covexillary permutation,
we have an open immersion $\tau\circ h\circ\iota: G_w\to Gr_{\underline v}$.
\end{cor}

\subsection{Equivariance of the Embedding}
The $B\times B$-action on \g\ restricts to a $T\times T$-action on \g.
Let $T_{2n}$ be the group of invertible $2n\times 2n$ diagonal matrices,
acting on $Gr(n,2n)$ via left-multiplication.
Let ${}^\tau\_:T_{2n}\to T_{2n}$ denote conjugation by the permutation $\tau$.
Making the identification,
\begin{align*}
T\times T\overset\sim\to T_{2n},&& (t_l,t_r)\mapsto \begin{pmatrix}t_r&0\\0&t_l\end{pmatrix},
\end{align*}
we have a commutative diagram of actions,
\begin{center}
\begin{tikzcd}
T\times T\arrow[r,"\sim"]                               & T_{2n}\arrow[r,"{}^\tau\_"]                     & T_{2n}\\
\g\arrow[r,hookrightarrow,"h"]\arrow[loop,looseness=5] & Gr(n,2n)\arrow[loop,looseness=3]\arrow[r,"\tau"] & Gr(n,2n)\arrow[loop,looseness=3].
\end{tikzcd}
\end{center}

\begin{rem}
Let $x_i$ (resp. $y_i$) denote the $i^{th}$ coordinate function on the second (resp. first) component of $T\times T$,
and let $t_i$ denote the $i^{th}$ coordinate function on $T_{2n}$.
The map $f:Hom(T_{2n},\IC^*)\to Hom(T\times T,\IC^*)$,
dual to the composite map
\begin{center}
\begin{tikzcd}
T\times T\arrow[r,"\sim"]& T_{2n}\arrow[r,"{}^\tau\_"]& T_{2n}.
\end{tikzcd}
\end{center}
is given by $t_{\tau(i)}\mapsto \begin{cases} y_i&\text{if }i\leq n\\ x_{i-n}&\text{if }i>n\end{cases}$.

Consider the double Schubert polynomial $\mathfrak S_{w_0w}$. 
Following \cite{MR2180402}, $\mathfrak S_{w_0w}$ is precisely the $T\times T$ multidegree of $\g_w$ in \g.
Moreover, since $\tau\circ h:\g\to Gr(n,2n)$ is an open immersion taking
$\g_w$ to $Gr_{\underline v}$ and $0\in\g$ to $\tau\in Gr(n,2n)$,
this multidegree equals the image under $f$ of $[Gr_{\underline v}]_{T_{2n}}|\tau$,
the localization at $\tau$ of the $T_{2n}$-equivariant cohomology class of $Gr_{\underline v}$.
\end{rem}

\subsection{Characteristic Cycles and Euler Obstructions}
The local Euler obstruction is a certain topological invariant of algebraic varieties.
It was first introduced by MacPherson in \cite{macpherson:chern},
and has seen many alternate definitions and interpretations,
see \cite[Ch. 8]{MR2574165} for a discussion.
Combinatorial models computing the local Euler obstructions of Grassmannian Schubert varieties were developed by \cite{boe.fu},
see also \cite{MR1084458,brylinski.kashiwara:KL}.

A long-standing problem in geometric representation theory is determining when the characteristic cycle of the IC sheaf of a Schubert variety $\Fl_w$ is irreducible,
and more generally, computing the characteristic cycle as a linear combination of conormal cycles.
Euler obstruction computations are useful in the study of the irreducibility of the characteristic cycle of the IC sheaf of a Schubert variety.
Precisely, the characteristic cycle is irreducible if and only if 
the Euler obstruction of $\Fl_w$ at a point $v$ equals 
the value of the Kazhdan-Lusztig polynomial \cite{KL:topological,MR573434} evaluated at $1$.
Bressler, Finkelberg, and Lunts \cite{MR1084458} used Zelevinsky's small resolutions to show that the characteristic cycle of a Grassmannian Schubert variety is irreducible.
Later, Boe and Fu \cite{boe.fu} gave a different proof,
showing that for Grassmannian Schubert varieties,
Lascoux and Sch\"utzenberger's combinatorial model \cite{MR1354702} for Kazhdan-Lusztig polynomials also computes the Euler obstructions.

Since both the Kazhdan-Lusztig polynomials and the Euler obstructions are local invariants,
the following is an immediate consequence of \cref{prop:MSVtoGr}.

\begin{thm}
\label{mainChar}
The characteristic cycle of the IC sheaf a covexillary Schubert variety is irreducible.
\end{thm}

\section{The Conormal Variety of a covexillary Matrix Schubert Variety}
In this section, we give a (possibly non-reduced) system of equations identifying the conormal variety $N^*\g_w$ of $\g_w$ in \g\ as a subvariety of the cotangent bundle $T^*\g$.
To this end, we recall standard descriptions of the cotangent bundles $T^*\g$ and $T^*Gr(n,2n)$,
and give a formula in \cref{hHash} for the induced map on cotangent bundles,
$\tau_{\#}\circ h_{\#}:T^*\g\to T^*Gr(n,2n)$, compatible with these descriptions.

\subsection{Conormal Varieties}
Let $M$ be a smooth variety, and $X\subset M$ a closed subvariety.
We recall the definition of the conormal variety $N^*X$
(or $N^*(X,M)$ when it is necessary to emphasize the ambient variety $M$)
of $X$ in $M$.
Let $X^{sm}$ be the smooth locus of $X$.
The conormal bundle $N^*X^{sm}\to X^{sm}$ is the vector bundle given by
\begin{align*}
N_x^*X^{sm}=\set{\alpha\in T_x^*M}{\alpha(v)=0\,\forall\,v\in T_xM},
\end{align*}
i.e., it is the vector bundle whose fibre at a point $x\in X^{sm}$ is the annihilator (in $T^*M$) of the tangent subspace $T_xX^{sm}$.
The closure of $N^*X^{sm}$ in $T^*M$ is called the conormal variety $N^*X$ of $X$ in $M$.

\subsection{Cotangent bundles and pull-backs}
\label{sec:TangentBundles}
Consider a smooth map $\gamma:D\to M$ to $M$ from some (Euclidean) neighborhood $D$ of $0\in\IC$.
Then $\gamma$ corresponds to a tangent vector at the point $\gamma(0)\in M$,
and every tangent vector can obtained in this manner.
We will call $\gamma$ a local curve in $M$.
Given a smooth map $f:M\to N$, we have an induced map $f_*:TM\to TN$ of tangent bundles.
If $v$ is the tangent vector corresponding to a curve $\gamma:D\to M$,
then $f_*(v)$ is the tangent vector corresponding to the curve $f\circ\gamma:D\to N$.

\label{sec:CotGen}
For $x\in M$, let $f_{*,x}:T_xM\to T_{f(x)}N$ denote the restriction of $f_*$ to the tangent space $T_xM$.
Let $f':f^*T^*N\to T^*N$ be the base change of the map $f:M\to N$ along the map $T^*N\to N$.
For $x\in M$, the fibre $(f^*T^*N)_x$ is precisely the cotangent space $T^*_{f(x)}N$.
Dual to the map $f_{*,x}$, we have a map $f^*_x:(f^*T^*N)_x\to T_x^*M$,
and hence a map $f^*:f^*T^*N\to T^*M$.
\begin{center}
\begin{tikzcd}
T^*M\arrow[dr] &\arrow[l,"f^*"']f^*T^*N\arrow[r,"f'"]\arrow[d]&T^*N\arrow[d]\\
&M\arrow[r,"f"]&N\arrow[ul,phantom,"\times"].
\end{tikzcd}
\end{center}

If $f:M\to N$ is an \emph{open immersion}, the map $f^*$ is an isomorphism,
and we have an open immersion $f_{\#}:T^*M\to T^*N$
given by $f_{\#}=f'\circ (f^*)^{-1}$.
\begin{center}
\begin{tikzcd}
T^*M\arrow[r,"(f^*)^{-1}"]\arrow[dr]&f^*T^*M\arrow[r,"f'"]\arrow[d]&T^*N\arrow[d]\\
&M\arrow[r,"f"]&N\arrow[ul,phantom,"\times"].
\end{tikzcd}
\end{center}
Equivalently, $T^*M$ is simply the restriction of $T^*N$ to $M$,
and $f_{\#}:T^*M\to T^*N$ is the inclusion map.

\subsection{Tangent and cotangent bundle of the Grassmannian}
\label{subCot}
Since $\g$ is a vector space, we have natural identifications
\begin{align}
\label{tangentM}
T\g=\g\times\g,&&T^*\g=\g\times\g^*.
\end{align}
The point $(x,y)\in \g\times \g=T\g$ corresponds to the local curve $t\mapsto x+ty$.

The trace form on $\g$ is a non-degenerate bilinear form,
and hence induces an isomorphism $\alpha:\g\to \g^*$.
For $y\in \g$, we write $\alpha_y\in \g^*$ for the linear functional given by $\alpha_y(x)=\tr(yx)$.

Let $E_{2n}$ be a $2n$-dimensional vector space with standard basis $e_1,\cdots,e_{2n}$,
and let $P\subset GL_{2n}$ denote the stabilizer of the subspace $E_{n}\subset E_{2n}$.
Let $\g_{2n}=Lie(GL_{2n})$, $\mathfrak p=Lie(P)$,
and let $\theta:\g\to\g_{2n}/\p$ denote the linear isomorphism
\begin{align*}
\theta(y)=
\begin{pmatrix}
0&0\\
y&0\\
\end{pmatrix}
(\mathrm{mod}\ \p),
\end{align*}
where the matrix on the right is in block form with blocks of size $n\times n$.
Recall the isomorphism $GL_{2n}/P=Gr(n,2n)$ given by $gP\mapsto gE_{n}$.
We have
\begin{align}
\label{GrassTangent}
T^*Gr(n,2n)=GL_{2n}\times^P(\g_{2n}/\p),
\end{align}
with $(g,\theta(y))\in GL_{2n}\times^P(\g_{2n}/\p)$ corresponding to the curve $t\mapsto g\,\Col{ty}$.

Let $\mathfrak u_P$ be the unipotent radical of $\mathfrak p$.
The trace form on $\g_{2n}$ induces a non-degenerate bilinear pairing $\tr:\g_{2n}/\p\times \mathfrak u_P\to \IC$.
Let $\theta^*:\g^*\to\mathfrak u_P$ denote the isomorphism
\begin{align*}
\theta^*(\alpha_y)=
\begin{pmatrix}
0&y\\
0&0\\
\end{pmatrix},
\end{align*}
where the matrices are in block form with blocks of size $n\times n$.
We have a commutative diagram,
\begin{center}
\begin{tikzcd}
\g\times \g^*\arrow[rr,"\theta\times\theta^*"]\arrow[dr]&&\g_{2n}/\p\times\mathfrak u_P\arrow[dl,"\tr"]\\
&\IC&
\end{tikzcd}
\end{center}
where the map $\g\times \g^*\to\IC$ is simply the evaluation map $(x,\alpha_y)\mapsto\alpha_y(x)=\tr(yx)$.

Dual to \cref{GrassTangent}, the cotangent bundle of $Gr(n,2n)$ is given by
\begin{align}
\label{cotAll}
T^*Gr(n,2n)=GL_{2n}\times^P\mathfrak u_P.
\end{align}

\begin{lemma}
\label{hHash}
Let $h_1:\g\to GL_{2n}$ be the map given by
\begin{align*}
h_1(x)=
\begin{pmatrix}
I(n)&0\\
x&I(n)\\
\end{pmatrix}.
\end{align*}
Recall from \cref{map:h} the graph embedding $h:\g\to Gr(n,2n)$. 
The map $h_{\#}:T^*\g\to T^*Gr(n,2n)$ is given by 
\begin{align*}
h_{\#}(x,\alpha_y)=(h_1(x),\theta^*(\alpha_y))\in GL_{2n}\times^P\mathfrak u_P.
\end{align*}
\end{lemma}
\begin{proof}
Consider the local curve $\gamma$ in \g\ given by $\gamma(t)=x+ty$.
We have
\begin{align*}
h\circ\gamma(t)=\Col{x+ty}=h_1(x) \Col{ty},
\end{align*}
We see that $h_*(x,y)=(h_1(x),\theta(y))\in GL_{2n}\times^P(\g_{2n}/\p)$,
and hence
\begin{align*}
h_{\#}(x,\alpha_y)=(h_1(x),\theta^*(\alpha_y))\in GL_{2n}\times^P\mathfrak u_P.
\end{align*}
\end{proof}

\begin{lemma}
Recall $\tau\in S_{2n}$ from \cref{tauMatrix},
and the associated  map $\tau:Gr(n,2n)\to Gr(n,2n)$ given by $V\mapsto \tau V$.
We have
$\tau_{\#}(g,x)=(\tau g,x)\in GL_{2n}\times^P\mathfrak u_P$.
\end{lemma}
\begin{proof}
At any point $gE_n\in Gr(n,2n)$, the pairing
\begin{align*}
T_{gE_n}Gr(n,2n)\times T^*_{gE_n}Gr(n,2n)\to\IC
\end{align*}
is given by $((g,y),(g,x))\mapsto\tr(yx)$.
The point $(g,x)\in G\times^P(\g/\p)$ corresponds
to the local curve $\gamma$ given by
$\gamma(t)=g Col\begin{pmatrix}I(n)\\ty\end{pmatrix}$.
We have $\tau\circ\gamma(t)=
\tau g Col\begin{pmatrix}I(n)\\ty\end{pmatrix}$,
and hence $\tau_*(g,y)=(\tau g,y)$.
It follows that $\tau_{\#}(g,x)=(\tau g,x)$.
\end{proof}

\subsection{The Springer map}
\label{sec:Springer}
The Springer map $\mu_{2n}:T^*Gr(n,2n)\to\g_{2n}$,
given by
\begin{align*}
&&\mu_{2n}(g,x)=Ad(g)x=gxg^{-1},&&(g,x)\in GL_{2n}\times^P\mathfrak u_P.
\end{align*}
yields a closed immersion,
\begin{align}
\label{grassEmbed}
j:GL_{2n}\times^P\mathfrak u_P\to\g_{2n},&&(g,x)\mapsto (gE_{n},gxg^{-1}),
\end{align}
cf. \cite{MR0263830} (see also \cite{MR1433132}).
We have 
\begin{align}
\label{grassSpringer}
j(T^*Gr(n,2n))=\set{(V,x)\in Gr(n,2n)\times\g_{2n}}{\mathrm{Im}(x)\subset V\subset\ker(x)}.
\end{align}

A system of defining equations for the conormal variety of any Grassmannian Schubert variety was developed in \cite{MR4257099}.
We present the result in a language conducive to our needs.

\begin{prop}
\label{grassCon}
{\rm (cf. \cite[Thm B]{MR4257099})}
Let $Gr_{\underline u}\subset Gr(d,N)$ be a Grassmannian Schubert variety given by the conditions
\begin{align*}
Gr_{\underline u}=\set{V\in Gr}{\dim(V+E_{t'_i})\leq d+c_i\ \forall\,i}.
\end{align*}
for some integers $1\leq t'_1\leq\cdots\leq t'_k\leq N-1$
and $0\leq c_1\leq\cdots\leq c_k\leq N-d$.
A point $(V,x)\in j(T^*Gr(d,N))$ is in $j(N^*Gr_{\underline u})$ if and only if $V\in Gr_{\underline u}$ and 
\begin{align*}
{\dim(xE_{t'_i}/E_{t'_j})\leq \begin{cases}(t'_{i-1}-c_{i-1})-(t'_j-c_j),\\ c_i-c_{j+1},\end{cases}\ \forall\,i}.
\end{align*}
\end{prop}

\subsection{The conormal variety of a covexillary matrix Schubert variety}
Let $w$ be a covexillary partial permutation, with essential set 
\begin{align*}
\ess(w)=\{(p_1+1,q_1),\cdots,(p_{m-1}+1,q_{m-1})\}
\end{align*}
for some
$0\leq p_1\leq\cdots\leq p_{m-1}\leq n-1$ and $1\leq q_1\leq\cdots\leq q_{m-1}\leq n$.
Let $r_i=\mathbf r^w_{ p_i-1,q_i }$,
\begin{align*}
Gr_{\underline v}&=\set{V\in Gr(n,2n)}{\dim(V+E_{t_i})\leq n+p_i+r_i},
\end{align*}
and let $\tau\in S_{2n}$
and $\tau\circ h:\g_w\hookrightarrow Gr_{\underline v}$ be as in \cref{prop:MSVtoGr}.
Applying \cref{grassCon} to $Gr_{\underline v}\subset Gr(n,2n)$,
we see that a point $(V,x)\in j(T^*Gr(n,2n))$
belongs to $j(N^*Gr_{\underline v})$ if and only if $V\in Gr_{\underline u}$
and 
\begin{align}
\label{grassCon2}
&&&&xE_{t_i}/E_{t_j}\leq
\begin{cases}
(q_{i-1}-r_{i-1})-(q_j-r_j)\\
(p_i+r_i)-(p_{j+1}+r_{j+1})\\
\end{cases}
&&\forall\,1\leq j<i<m.
\end{align}

\begin{thm}
\label{thm:conMat}
Consider the $2n\times 2n$ matrix $M$ given in block form by
\begin{align*}
M=
\begin{pmatrix}
yx  & y\\
xyx & xy
\end{pmatrix},
\end{align*}
and let $M(i,j)$ be the submatrix of $M$ spanned by the rows $\{q_j+1,\cdots,n,n+p_j+1,\cdots,2n\}$ and the columns $\{1,\cdots,q_i,n+1,\cdots n+p_i\}$.
We have $(x,\alpha_y)\in N^*\g_w$ if and only if $x\in\g_w$ and
\begin{align*}
&&\rk(M(i,j))\leq
\begin{cases}
(q_{i-1}-r_{i-1})-(q_j-r_j),\\
(p_i+r_i)-(p_{j+1}+r_{j+1}).\\
\end{cases}
&&\forall\, 1< j\leq i< m.
\end{align*}
\end{thm}
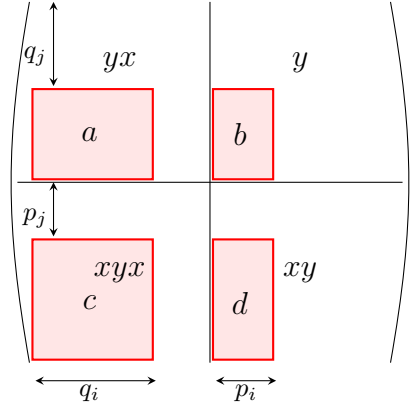
\begin{figure}[ht]
\begin{tikzpicture}[scale=0.8]
\draw (0,0)to[out=100,in=-100](0,6);
\draw (6,0)to[out=80,in=-80](6,6);
\draw (-0.2,3)--(6.2,3);
\draw (3,0)--(3,6);
\draw[red,thick,fill=red!10] (0.05,0.05) rectangle ++(2,2);
\draw[red,thick,fill=red!10] (0.05,3.05) rectangle ++(2,1.5);
\draw[red,thick,fill=red!10] (3.05,3.05) rectangle ++(1,1.5);
\draw[red,thick,fill=red!10] (3.05,0.05) rectangle ++(1,2);
\draw (1,-0.5) node {$q_i$};
\draw (0.1,-0.3)[stealth-stealth]--(2.1,-0.3);
\draw (3.6,-0.5) node {$p_i$};
\draw (3.1,-0.3)[stealth-stealth]--(4.1,-0.3);
\draw (0.1,2.4) node {$p_j$};
\draw (0.4,3)[stealth-stealth]--(0.4,2.1);
\draw (0.1,5.1) node {$q_j$};
\draw (0.4,6)[stealth-stealth]--(0.4,4.6);

\draw(1.5,1.5) node {\Large$xyx$};
\draw(4.5,1.5) node {\Large$xy$};
\draw(4.5,5) node {\Large$y$};
\draw(1.5,5) node {\Large$yx$};
\draw (1,3.8) node {\Large$a$};
\draw (3.5,3.8) node {\Large$b$};
\draw (1,1) node {\Large$c$};
\draw (3.5,1) node {\Large$d$};
\end{tikzpicture}
\caption{The matrix $M(i,j)$ is the shaded portion of the block matrix $M=\begin{pmatrix} yx&y\\xyx&xy\end{pmatrix}$.}
\end{figure}

\begin{proof}
The map $\tau\circ h:\g_w\to Gr_{\underline v}$ is an open immersion,
hence we have the following Cartesian square,
\begin{center}
\begin{tikzcd}
N^*\g_w\arrow[d]\arrow[dr,phantom,"\times"]\arrow[r,hook,"\tau_{\#}\circ h_{\#}"]\arrow[d] & N^*Gr_{\underline v}\arrow[d]\\
\g_w \arrow[r,hook,"\tau\circ h"]                                                          & Gr_{\underline v}.
\end{tikzcd}
\end{center}
It follows that for $(x,\alpha_y)\in T^*\g$,
we have 
\begin{align*}
(x,\alpha_y)\in N^*\g_w&\iff\tau_{\#}\circ h_{\#}(x,\alpha_y)\in N^*Gr_{\underline v}\\
&\iff j\circ\tau_{\#}\circ h_{\#}(x,\alpha_y)\in j(N^*Gr_{\underline v}).
\end{align*}
Following \cref{iHash,hHash}, we see that
\begin{align*}
j\circ \tau_{\#}\circ h_{\#}(x,\alpha_y)
&=j(\tau h_1(x),\theta^*(y))\\
&=j\left(\tau
\begin{pmatrix}
I(n)&0\\
x&I(n)\\
\end{pmatrix},
\begin{pmatrix}
0&y\\
0&0\\
\end{pmatrix}
\right)\\
&=
\left(\tau \Col x,
\tau
\begin{pmatrix}
-yx  & y\\
-xyx & xy\\
\end{pmatrix}
\tau^{-1}
\right)\\
\end{align*}
Following \cref{grassCon2}, we have $j\circ\tau_{\#}\circ h_{\#}(x,\alpha_y)\in j(N^*Gr_{\underline v})$ if and only if
\begin{align}
\label{workTau}
\dim\left(\tau
\begin{pmatrix}
-yx  & y\\
-xyx & xy\\
\end{pmatrix}
\tau^{-1}\e{t_i}/\e{t_j}\right)\leq
\begin{cases}
(q_{i-1}-r_{i-1})-(q_j-r_j),\\
(p_i+r_i)-(p_{j+1}+r_{j+1}).\\
\end{cases}
\end{align}
Next, observe that
\begin{align*}
\dim\left(\tau
\begin{pmatrix}
-yx  & y\\
-xyx & xy\\
\end{pmatrix}
\tau^{-1}\e{t_i}/\e{t_j}\right)
&=
\dim\left(
\begin{pmatrix}
-yx  & y\\
-xyx & xy\\
\end{pmatrix}
\tau^{-1}\e{t_i}/\tau^{-1}\e{t_j}\right)\\
&=
\dim\left(
M\tau^{-1}\e{t_i}/\tau^{-1}\e{t_j}\right)
\end{align*}
Recall from \cref{tauInverseE} that
$
\tau^{-1}\e{t_i}=\left\langle e_1,\cdots,e_{q_i},e_{n+1},\cdots,e_{n+p_i}\right\rangle.
$
We deduce that the columns of the matrix $M(i,j)$ form a basis of 
$
M \tau^{-1}\e{t_i}/\tau^{-1}\e{t_j}$,
and hence the result follows from \cref{workTau}.
\end{proof}

\section{The conormal variety of a covexillary Schubert variety}
Let $\iota:G\to \g$ denote the inclusion of $GL_n$ into the set of $n\times n$ matrices.
Recall the map $\pi:G\twoheadrightarrow\Fl$ given by $g\mapsto gE_\bullet$,
and the corresponding Cartesian square,
\begin{equation}
\label{piStarMap}
\begin{tikzcd}
\pi^*T^*\Fl\arrow[d]\arrow[r,"\pi'",->>] & T^*\Fl\arrow[d]\\
G\arrow[r,"\pi",->>]                     & \Fl\arrow[ul,phantom,"\times"].
\end{tikzcd}
\end{equation}
Since $\pi:G\to\Fl$ is a smooth morphism,
for any subvariety $X\subset\Fl$,
we have $\pi'^{-1}(N^*X)=N^*(\pi^{-1}(X))$,
see \cite[p. 65]{HTT}.
Applied to $X=\Fl_w$,
this yields the following diagram of Cartesian squares,
\begin{equation}
\label{bigCommDiag}
\begin{tikzcd}
N^*\Fl_w\arrow[d]\arrow[dr,phantom,"\times"]     & N^*G_w\arrow[d]\arrow[l,"\pi'"',->>]\arrow[dr,phantom,"\times"]\arrow[r,hook,"\iota_{\#}"]\arrow[d] & N^*\g_w\arrow[d]\\
\Fl_w\arrow[d,hook]  \arrow[dr,phantom,"\times"] & \arrow[l,"\pi"',->>]G_w\arrow[d,hook] \arrow[dr,phantom,"\times"]\arrow[r,hook,"\iota"]                  & \g_w\arrow[d,hook]\\
\Fl                                              & G\arrow[l,"\pi"',->>]\arrow[r,hook,"\iota"]                                                              & \g.
\end{tikzcd}
\end{equation}

In this section, we use \cref{thm:conMat,bigCommDiag} to give a characterization of $N^*\Fl_w$
when $w$ is a covexillary permutation.

Let $w\in S_n$ be a covexillary permutation,
and let $\tau\in S_{2n}$ be as in \cref{tauMatrix}.
Recall that $\g$ is naturally identified with the tangent space of $G$ at identity;
the point $x\in \g$ corresponds to the local curve $t\mapsto 1+tx$.
By $G$-equivariance, we have 
$
TG=G\times \g,
$
with the point $(g,x)\in G\times \g$ corresponding to the curve $t\mapsto g(1+tx)$.
Dually, we have $T^*G=\g\times\g^*$.
Using the trace form on \g,
we have an isomorphism $\alpha:\g\to\g^*$,
given by $\alpha_y=(x\mapsto\tr(yx))$,
see \cref{subCot}.

\begin{lemma}
\label{iHash}
Consider the open immersion $\iota:G\to \g$.
We have $\iota_{\#}(g,\alpha_y)=(g,\alpha_{yg^{-1}})$.
\end{lemma}
\begin{proof}
Recall from \cref{sec:TangentBundles} the induced map $\iota_*:TG\to T\g$ on tangent bundles.
Consider the local curve $\gamma$ given by $\gamma(t)=g(1+tx)$;
we have $\iota\circ\gamma(t)=g+tgx$,
and hence $\iota_*(g,x)=(g,gx)$.

Since $\alpha_y(gx)=\tr(ygx)=\alpha_{yg}(x)$,
the dual map $\iota^*_g:T^*_g\g\to T^*_gG$ is given by $\alpha_y\mapsto\alpha_{yg}$.
Following \cref{sec:CotGen}, $\iota_{\#,g}:T^*_gG\to T^*_g\g$ is the inverse of the isomorphism $\iota^*_g$,
and hence is given by $\iota_{\#,g}(\alpha_y)=\alpha_{yg^{-1}}$.
It follows that $\iota_{\#}(g,\alpha_y)=(g,\alpha_{yg^{-1}})$.
\end{proof}

Let \bo\ be the Lie algebra of $B$,
and let $\mathfrak u_B$ be the unipotent radical of \bo.
Following the isomorphism $G/B\overset\sim\to\Fl$ given by $gB\mapsto gE_\bullet$,
we see that $T_{E_\bullet}\Fl=\g/\bo$,
and hence by $G$-equivariance, $T\Fl=G\times^B\g/\bo$.
The trace form on $\g$ induces a non-degenerate bilinear form $\g/\bo\times\mathfrak u_B\to\IC$,
and hence we have the identifications $\mathfrak u=(\g/\bo)^*$ and $T^*\Fl=G\times^B\mathfrak u$.

Recall the Springer map $\mu:T^*G/B\to\g$, given by $(g,y)\mapsto gyg^{-1}$.
Following \cite{MR0263830} (see also \cite{MR1433132}), we have an embedding
\begin{align}
\label{SpringerFlag}
j:T^*\Fl\hookrightarrow\Fl\times\g,&&
(g,y)\mapsto(gE_\bullet,gyg^{-1}).
\end{align}
The image of $T^*\Fl$ under this map has the following description:
\begin{equation*}
j(T^*\Fl)=\set{(F_\bullet,z)\in\Fl\times\g}{zF_i\subset F_{i-1}\ \forall\,i}.
\end{equation*}

\begin{lemma}
We have an identification $\pi^*T^*\Fl=G\times\mathfrak u_B$,
with the inclusion $\pi^*T^*\Fl\hookrightarrow T^*G$ given by $(g,y)\mapsto (g,\alpha_y)$.
Moreover, the map $\pi':\pi^*T^*\Fl\to T^*\Fl$ (see \cref{piStarMap})
is the quotient map
\begin{align*}
G\times\mathfrak u_B\to G\times^B\mathfrak u_B,&&
(g,y)\mapsto (g,y)(\mathrm{mod}\ B).
\end{align*}
\end{lemma}
\begin{proof}
It is clear that we have Cartesian squares,
\begin{center}
\begin{tikzcd}
G\times\mathfrak u_B\arrow[r]\arrow[d]&G\times^B\mathfrak u_B\arrow[d]\\
G\arrow[r,"\pi"]&\Fl\arrow[lu,phantom,"\times"],
\end{tikzcd}
\qquad\qquad
\begin{tikzcd}
\pi^*T^*\Fl\arrow[r]\arrow[d]&G\times^B\mathfrak u_B\arrow[d]\\
G\arrow[r,"\pi"]&\Fl\arrow[lu,phantom,"\times"].
\end{tikzcd}
\end{center}
By the uniqueness of fibre products, we deduce an isomorphism identifying the pull-back $\pi^*T^*\Fl$ with $G\times\mathfrak u_B$.
Further, the map $(g,y)\mapsto (g,\alpha_y)$
sends ${g}\times\mathfrak u_B$ to the annihilator of $T_g(gB)$ (in $T^*_gG$),
and hence can be identified with the inclusion $\pi^*T^*\Fl\hookrightarrow T^*G$.
\end{proof}

\begin{thm}
\label{thm:conCov}
Consider $(F_\bullet,z)\in j(T^*\Fl)$.
We have $(F_\bullet,z)\in j(N^*\Fl_w)$ if and only if $F_\bullet\in \Fl_w$ and 
\begin{align}
\label{work4}
\dim(z(F_{q_i}+E_{p_i})/(F_{q_j}\cap E_{p_j}))\leq
\begin{cases}
(q_{i-1}-r_{i-1})-(q_j-r_j)\\
(p_i+r_i)-(p_{j+1}+r_{j+1})\\
\end{cases}
\end{align}
for all $1\leq i<j<m$.
\end{thm}

\begin{proof}
Consider the commutative diagram (see \cref{bigCommDiag}),
\begin{center}
\begin{tikzcd}
      & N^*G_w\arrow[dl,hook']\arrow[r,->>,"\pi'"]\arrow[d,hook]\arrow[dr,phantom,"\times"] & N^*\Fl_w\arrow[d,hook]\arrow[dr,hook]&\\
T^*G  & \arrow[l,hook']\pi^*T^*\Fl\arrow[r,->>,"\pi'"]                                      & T^*\Fl\arrow[r,hook,"j"]             & \Fl\times\g.
\end{tikzcd}
\end{center}
Since the induced map $\pi':N^*G_w\to N^*\Fl_w$ is surjective,
we have $(F_\bullet,z)\in j(N^*\Fl_w)$ if and only if
there exists $(g,y)\in N^*G_w$ such that
\begin{align*}
j\circ\pi'(g,y)=(gE_\bullet,gyg^{-1})=(F_\bullet,z). 
\end{align*}
Following \cref{iHash,bigCommDiag},
we see that $(g,y)\in N^*G_w$ if and only if $\iota_{\#}(g,y)=(g,\alpha_{yg^{-1}})\in N^*\g_w$. 
Consider the matrix $M$ given in block form by
\begin{align*}
M=
\begin{pmatrix}
y  & yg^{-1}\\
gy & gyg^{-1}\\
\end{pmatrix}
=
\begin{pmatrix}
g^{-1}zg & g^{-1}z\\
zg       & z\\
\end{pmatrix},
\end{align*}
and let $M(i,j)$ be the submatrix of $M$ spanned by the rows $\{q_j+1,\cdots,n,n+p_j+1,\cdots,2n\}$ and the columns $\{1,\cdots,q_i,n+1,\cdots n+p_i\}$.
Following \cref{thm:conMat}, we have $(g,\alpha_{yg^{-1}})\in N^*G_w$ if and only if $g\in G_w$ and
\begin{align}
\label{work3}
&&&&\rk(M(i,j))\leq
\begin{cases}
(q_{i-1}-r_{i-1})-(q_j-r_j)\\
(p_i+r_i)-(p_{j+1}+r_{j+1})\\
\end{cases}
&&\forall\,1\leq i<j<m.
\end{align}
We finish the proof by showing that \cref{work3,work4} are equivalent. 

Consider the automorphism $\widetilde g$ of $\e n\oplus\e n$ given by $(v,w)\mapsto(gv,w)$.
Observe that $\widetilde g(E_{q_j}\oplus E_{p_j})=F_{q_j}\oplus E_{p_j}$,
and hence $\widetilde g$ descends to a map $\widetilde g_0$ satisfying the commutative diagram,
\begin{center}
\begin{tikzcd}
\e n\oplus\e n\arrow[r,"\widetilde g"]\arrow[d,"\eta"]&\e n\oplus\e n\arrow[d,"\eta'"]\\
(\e n\oplus\e n)/(\e{q_j}\oplus\e{p_j})\arrow[r,"\widetilde g_0"]&(\e n\oplus\e n)/(F_{q_j}\oplus \e{p_j}),
\end{tikzcd}
\end{center}
where $\eta$, $\eta'$ are the quotient maps.

Let $\kappa:\e n\to \e n\oplus\e n$ denote the diagonal embedding $v\mapsto(v,v)$.
Observe that
\begin{align*}
&&\begin{pmatrix}g^{-1}z\\z\end{pmatrix}
&&\text{and}&&
&&\begin{pmatrix}g^{-1}zg\\zg\end{pmatrix}
\end{align*}
are precisely the matrices of the maps $\widetilde g\kappa z$ and $\widetilde g\kappa z g$ respectively
(with respect to the ordered bases $(e_1,\cdots,e_{n})$ of \e n,
and $((e_1,0),\cdots,(e_n,0),(0,e_1),\cdots,(0,e_n))$ of $\e n\oplus\e n$).
Moreover, removing the columns indexed $\{1,\cdots,q_j,n+1,\cdots,n+p_j\}$ from $M$
corresponds to quotienting $E_n\oplus E_n$ by the subspace $E_{q_j}\oplus E_{p_j}$.

It follows that first $q_i$ columns of $M(i,j)$ span the subspace
\begin{align*}
\eta\widetilde g^{-1}\kappa z gE_{q_i}=\eta\widetilde g^{-1}\kappa z F_{q_i}=\widetilde g_0^{-1}\eta' \kappa zF_{q_i},
\end{align*}
and the last $p_i$ columns of $M(i,j)$ span the subspace
\begin{align*}
\eta\widetilde g^{-1}\kappa z E_{p_i}=\eta\widetilde g^{-1}\kappa z E_{p_i}=\widetilde g_0^{-1}\eta' \kappa zE_{p_i}.
\end{align*}
Consequently, we have
\begin{align*}
\rk(M(i,j))&=\dim(\widetilde g_0^{-1}\eta'\kappa z(F_{q_i}+E_{p_i}))\\
&=\dim(\eta'\kappa z(F_{q_i}+E_{p_i}))\\
&=\dim(\kappa z(F_{q_i}+E_{p_i})/(F_{q_j}\oplus E_{p_j})).
\end{align*}
Finally, since $\kappa:E_n\to E_n\oplus E_n$ is the diagonal embedding,
we have
\begin{align*}
\operatorname{Im}(\kappa)\cap (F_{q_j}\oplus E_{p_i})=\kappa(F_{q_j}\cap E_{p_i}),
\end{align*}
and hence
\begin{align*}
\kappa z(F_{q_i}+E_{p_i})\cap (F_{q_j}\oplus E_{p_i})=
\kappa(z(F_{q_i}+E_{p_i})\cap F_{q_j}\cap E_{p_i}).
\end{align*}
It follows that
\begin{align*}
\dim(z(F_{q_i}+E_{p_i})/(F_{q_j}\cap E_{p_j}))&=\dim(\kappa z(F_{q_i}+E_{p_i})/(F_{q_j}\oplus E_{p_j}))\\
&=\rk(M(i,j)),
\end{align*}
and hence \cref{work4} is equivalent to \cref{work3}.
\end{proof}

\bibliographystyle{alpha}
\bibliography{biblio}

\end{document}